\documentclass[10pt]{amsart}
\usepackage[mathscr]{eucal}
\usepackage{amssymb,amsmath}
\usepackage{amsfonts}
\usepackage{a4}
\usepackage{enumerate}

\usepackage{graphicx,graphics}
\usepackage[all]{xy}

\numberwithin{equation}{section}

\theoremstyle{plain}
\newtheorem{thm}{Theorem}[section]
\newtheorem{prp}[thm]{Proposition}
\newtheorem{cor}[thm]{Corollary}
\newtheorem{lem}[thm]{Lemma}

\newtheorem{que}[thm]{Question}
\newtheorem{question}[thm]{Question}
\newtheorem{proposition}[thm]{Proposition}

\theoremstyle{remark}
\newtheorem{remark}[thm]{Remark}
\newtheorem{definition}[thm]{Definition}

\numberwithin{equation}{section}

\newcommand{\lin}{\operatorname{span}}
\newcommand{\wn}{\operatorname{Int}}

\begin{document}

\title[A \textsf{CDH} topological vector space is a Baire space]{A countable dense homogeneous topological vector space is a Baire space}
\author{Tadeusz Dobrowolski}

\address{
Department of Mathematics\\ Pittsburg State University\\
Pittsburg, KS 66762\\ USA} \email{tdobrowolski@pittstate.edu}

\author{Miko\l aj Krupski}

\address{Institute of Mathematics\\
University of Warsaw\\ Banacha 2\newline 02--097 Warszawa\\
Poland}
\email{mkrupski@mimuw.edu.pl}

\author{Witold Marciszewski}

\address{Institute of Mathematics\\
University of Warsaw\\ Banacha 2\newline 02--097 Warszawa\\
Poland} \email{wmarcisz@mimuw.edu.pl}

\date{\today}
\subjclass[2010]{Primary: {54C35}, {54E52}, {46A03}, Secondary: {22A05}} \keywords{Function space; pointwise convergence topology; $C_p(X)$ space; countable dense homogeneity;
Baire space; topological vector space; the property (B)}

\begin{abstract} We prove that every homogeneous countable dense homogeneous topological space containing a copy of the Cantor set is a Baire space.
In particular, every countable dense homogeneous topological vector space is a Baire space.
It follows that, for any nondiscrete metrizable space $X$, the function space $C_p(X)$ is not countable dense homogeneous. This answers a question
posed recently by R.\ Hern\'andez-Guti\'errez.
We also conclude that, for any infinite dimensional Banach space $E$ (dual Banach space $E^\ast$), the space $E$ equipped with the weak topology
($E^\ast$ with the weak$^\ast$ topology) is not countable dense homogeneous.  We generalize some results of Hru\v{s}\'ak, Zamora Avil\'es, and Hern\'andez-Guti\'errez concerning countable dense homogeneous products.
\end{abstract}

\maketitle

\section{Introduction}
In this note we only consider Tikhonov spaces.

 A space $X$ is \emph{countable dense
homogeneous} (\textsf{CDH}) if $X$ is separable and given countable
dense subsets $D,D' \subseteq X$ , there is a homeomorphism $h: X\to
X$ such that $h[D] = D'$. This is a classical notion going back to
works of Cantor, Fr\'echet and Brouwer. Among Polish spaces,
canonical examples of \textsf{CDH} spaces include the Cantor set,
Hilbert cube, space of irrationals, and all separable complete
metric linear spaces (m.l.s.) and manifolds modeled on them. In
contrast, every Borel but not closed vector subspace of a complete m.l.s. is
not \textsf{CDH}.
In recent years many efforts have been put in constructing examples of \textsf{CDH} non-Polish spaces (see the excellent survey articles
\cite[Section 14]{AvM} and
\cite{HvM}). With no extra algebraic structure, there even
exist \textsf{CDH} metrizable spaces that are meager (this notion is recalled below), e.g., see \cite{FZ} and \cite{FHR}.
An easy example of a non-metrizable \textsf{CDH} space is the Sorgenfrey line.

The main purpose of this work is examining the link between
(\textsf{CDH}) and Baireness (and also hereditary
Baireness) of a space, mainly, of a topological vector space. Recall
that a space $X$ is a \textit{Baire space} if the Baire Category Theorem
holds for $X$; the latter means that every sequence $(U_n)$ of dense
open subsets of $X$ has a dense intersection in $X$. If every closed
subset of $X$ is a Baire space then we call $X$ a \textit{hereditarily Baire
space}. A space is \textit{meager} if it can be written as a countable union of closed sets with empty interior.
Clearly, if $X$ is a Baire space, then $X$ is not meager. The reverse implication is in general not true. However, it holds for every homogeneous space $X$
(see \cite[Theorem 2.3]{LM}).
The idea of studying the relationship between \textsf{CDH} and (hereditarily) Baireness is not new and goes back to Fitzpatrick and Zhou's paper \cite{FZ}.
They showed, among other things, that if a homogeneous \textsf{CDH} metric space $X$ contains a countable set which is not $G_\delta$, then $X$
must be a Baire space.
Later the topic was studied by Hru\v{s}\'ak and Zamora Avil\'es \cite{HZA} who proved that that every
analytic metric \textsf{CDH} space must be a hereditarily Baire space.
Our paper deals mainly with topological vector spaces.

Function spaces $C_p(X)$ provide a wide class of topological vector
spaces to be investigated for \textsf{CDH}. By $C_p(X)$ we denote
the space of all continuous real-valued functions on a space $X$,
endowed with the pointwise topology.
V.\ Tkachuk has asked if there
exists a nondiscrete space $X$ such that $C_p(X)$ is \textsf{CDH}.
(For the case of discrete $X$ see below). Last year, R.\
Hern\'andez-Guti\'errez \cite{HG} gave the first consistent example
of such a space $X$. He has asked whether a metrizable
space $X$ must be discrete, provided $C_p(X)$ is \textsf{CDH} \cite[Question 2.6]{HG}. Our work was inspired by his paper;
in particular, we answer his question in the affirmative. Actually,
combining our Theorem \ref{thm_inf_bounded} with earlier results, we prove that, for a
metrizable space $X$, the space $C_p(X)$ is \textsf{CDH} if and only if $X$ is
discrete of cardinality less than the pseudointersection number
$\mathfrak{p}$.

In order to describe the example of Hern\'andez-Guti\'errez we need
to recall some definitions and older results. Given a filter $F$ on
an infinite countable set $T$, we regard $F$ as a subspace of $2^T$,
a topological copy of the Cantor set. We consider only free filters
on $T$, i.e., filters containing all cofinite subsets of $T$. By
$N_F$ we denote the space $T\cup\{\infty\}$, where $\infty\not\in
T$, equipped with the following topology: All points of $T$ are
isolated and the family $\{A\cup\{\infty\}: A\in F\}$ is a
neighborhood base at $\infty$. Recall that $F$ is a $P$-filter if
for every sequence $(U_n)$ of elements of $F$ there exists an $A\in
F$ which is {\it almost contained} in every $U_n$, i.e.\ $A\setminus
U_n$ is finite for every $n$. An ultrafilter (that is, a maximal
filter) that is a $P$-filter is also called a $P$-point.

Hereditarily Baire spaces $C_p(N_F)$ have been characterized in
terms of $F$ by the third named author (see, \cite[Thm.\ 1.2]{Ma}).

\begin{thm}
Let $F$ be a filter on $\omega$. The following are equivalent:
\begin{enumerate}[(a)]
\item  $F$ is a nonmeager $P$-filter;
\item $F\subset2^\omega$ is a hereditarily Baire space;
\item  $C_p(N_F)$ is a hereditary Baire space.
\end{enumerate}
\end{thm}

Hern\'andez-Guti\'errez, Hru\v{s}\'ak \cite{HGH}; and Kunen, Medini,
Zdomskyy \cite{KMZ} added the following equivalent condition to this
list {\em
\begin{enumerate}[(d)]
\item  $F\subset2^\omega$ is \textsf{CDH}.
\end{enumerate}}
Hern\'andez-Guti\'errez \cite{HG} proved that the above condition is
equivalent to {\em
\begin{enumerate}[(e)]
\item $C_p(N_F)$ is \textsf{CDH}.
\end{enumerate}}

Recall that it is unknown whether nonmeager $P$-filters exist in
ZFC. In particular, the following question remains open.

\begin{question}
Does there exist in ZFC a \textsf{CDH} topological vector space which is not Polish?
\end{question}

The equivalence of conditions (a)--(e) above inspired us to
investigate the relationship between the \textsf{CDH} and
(hereditarily) Baire properties of topological vector spaces $E$. We
prove that, if such a space $E$ is \textsf{CDH} then it is a Baire
space, cf. Corollary \ref{CDH_implies_Baire_tvs}. In the last section, we
provide examples of hereditarily Baire separable pre-Hilbert spaces
which are not \textsf{CDH}. Consistently, there exist \textsf{CDH}
topological vector spaces which are not hereditarily Baire, cf.\ the
comment after Theorem \ref{characterization_metr}.

Our techniques apply to Cartesian products. We show that for any separable space $X$ and any infinite cardinal number $\kappa$, if the product $X^\kappa$ is
\textsf{CDH}, then $X$ must be a Baire space, cf. Theorem \ref{product_Baire}. This generalizes a theorem of Hru\v{s}\'ak and Zamora Avil\'es
\cite[Theorem 3.1]{HZA}, who proved the same result for the countable power of a separable metric space $X$. We also show that
if the product $X\times Y$ of two crowded spaces $X$ and $Y$ is \textsf{CDH}, then $X$ contains a copy of the Cantor set if and only if $Y$ contains a copy of the
Cantor set, cf. Theorem \ref{product_Cantor_set}. This generalizes a theorem of Hern\'andez-Guti\'errez \cite[Theorem 2.3]{HG1}, who obtained the same result
under
additional assumption that both spaces $X$ and $Y$ have countable $\pi$-weight.

\section{\textsf{CDH} and property (B) of topological vector spaces}
Some spaces that we discuss are not \textsf{CDH} because they contain nonhomeomorphic countable dense subspaces.
Therefore, it is reasonable to present the following term: A space
$X$ is \emph{uniquely separable} (\textsf{US}) if $X$ is separable
and any two countable dense subspaces of $X$ are homeomorphic.\footnote{This property was considered by Arhangel'skii and van Mill in \cite{AvM1} under the name of
$c_1$-space.}
Obviously, for any $X$, \textsf{CDH} implies \textsf{US}.

Recall that a space $Y$ is referred to as {\it crowded} if every
neighborhood of any point contains at least two distinct points. We
will use the following stronger notion: The space $Y$ is
\textit{sequentially crowded} if every point $y\in Y$ is the limit
of a sequence of points of $Y$ distinct from $y$.

The classic characterization of the space of rationals $\mathbb{Q}$
yields

\begin{remark}
Every crowded separable metrizable space is \textsf{US}.
\end{remark}

It turns out that for certain topological vector spaces the lack of
having \textsf{CDH} is manifested drastically by not even being
\textsf{US}. That happens for spaces with property (B).

\begin{definition}
A space $X$ has the \textit{property (B)} provided $X$ can be
covered by countably many closed nowhere-dense sets
$\{A_n:n\in\omega\}$ such that for any compact set $K\subseteq X$
there exists $n\in \omega$ with $K\subseteq A_n$.
\end{definition}

This notion was introduced in \cite{KM} and later studied by Tkachuk
\cite{T1} and \cite{T2} under the name of \textit{Banakh property}.
Here is the main result of this section. It states, in particular, that no
topological vector space can simultaneously have the property (B)
and be \textsf{US}.

\begin{thm}\label{US_no_B} If a topological group $G$ has the property (B) and contains a nontrivial convergent sequence, then $G$ is
not \textsf{US} (hence, is not \textsf{CDH}).
In particular, if a topological vector space $E$ has the property (B), then $E$ is
not \textsf{US}.
\end{thm}

This theorem follows immediately from the next two lemmas and the trivial observation that a topological vector space $E$ with the property (B) must contain a
nonzero vector $x$, hence contains a nontrivial convergent sequence $(x/n)_{n\in\omega}$.

\begin{lem}\label{seq_crowded}
Any separable homogeneous space $X$ with a
nontrivial convergent sequence contains a sequentially crowded
countable dense subset. Furthermore, if, in addition $X$ is a topological group, such a set can be taken as a subgroup of the group.
\end{lem}
\begin{proof}
First, we will present a proof for the case of a topological group.
Let $(x_n)$ be a nontrivial convergent sequence in $X$. We may
assume that $\lim x_n$ is the neutral element of $X$. Take any countable dense subset $Y$ of $X$. Without loss of generality, we may further assume that all $x_n$'s are in $Y$. Then $D$, the smallest
subgroup of $X$ that contains $Y$, has the required properties.

Now, assume that $X$ is a separable homogeneous space with a nontrivial convergent sequence. By homogeneity, each point $x\in X$ is the limit
of a sequence of points of $X$ distinct from $x$. Therefore, for each countable subset $A$ of $X$,
there is a countable set $B$ containing $A$, such that each point $x\in A$ is the limit of a sequence of points of $B$ distinct from $x$.
Take any countable dense subset $D_0$ of $X$. Inductively, choose countable sets $D_0\subset D_1\subset D_2\subset\dots$, such that,
for any $n\in\omega$ and $x\in D_n$, there is a  sequence of points of $D_{n+1}$, distinct from $x$, and, converging to $x$.
One can easily verify that the union $\bigcup_{n\in \omega} D_n$ is the required dense subset.
\end{proof}

\begin{lem}\label{lemma_B_not_US}
Let $X$ be a space with a sequentially crowded
countable dense subset. If $X$ has the property (B), then it is not \textsf{US}.
\end{lem}
\begin{proof} Striving for a contradiction suppose that $X$ is
\textsf{US}. Then, every countable dense subset of $X$ is sequentially crowded. Let $\{A_n:n=1,2\ldots\}$ be a countable family
consisting of closed nowhere dense subsets of $X$ witnessing the
property (B). We will show that there is a convergent sequence
$C\subseteq X$ such that no $A_n$ contains $C$, yielding a
contradiction. Without loss of generality, we can assume that for
every $n$ we have $A_n\subseteq A_{n+1}$. Fix a countable dense set
$D_0\subseteq X$ and let $D_0=\{d_0,d_1,d_2,\ldots\}$ be a faithful
enumeration of $D_0$. Since the set $A_n$ is closed and has empty
interior, for each $n\geq 1$ there is a countable set $D_n\subseteq
X\setminus A_n$ dense in $X$. The set $D_n\cup\{d_n\}$ is
sequentially crowded being countable dense in $X$, hence there is a
sequence $(x^k_n)_{k\in \omega}$ of elements of $D_n$ converging to
$d_n$. Define
$$A=\{d_0\}\cup \bigcup_{n=1}^\infty\{x^k_n:k\in \omega\}.$$
The countable set $A$ is dense in $X$, because its closure contains
$D_0$, and hence it is sequentially crowded. Pick a sequence
$(a_m)_{m\in \omega}$ of elements of $A$, convergent to $d_0$. The
set $$C=\{a_m:m\in\omega\}\cup \{d_0\}$$ is as desired. Indeed, note
that for each $n\geq 1$ the set $C\cap \{x^k_n:k\in\omega\}$ is
finite, for otherwise $d_n$ would be an accumulation point of $C$
distinct from $d_0$. Therefore, $C\cap
\{x^k_n:k\in\omega\}\neq\emptyset$, for infinitely many $n$'s which
means that for infinitely many $n$'s $C$ meets the set $X\setminus
A_n$. Since the family $\{A_n:n=1,2,\ldots\}$ is increasing, no
$A_n$ contains $C$.
\end{proof}

It is easy to see, cf. \cite[page 653]{KM} that any
infinite-dimensional Banach spaces endowed with the weak topology
enjoys the property (B). More generally, we have the following:

\begin{proposition}\cite[Proposition 5.2]{KM}\label{B_weak_general}
Let $(E,\|\cdot\|)$ be a normed space and let $\tau$ be a linear topology on $E$, strictly weaker than the norm topology. 
If norm closed balls in $E$ are $\tau$-closed and $\tau$-compact
sets are norm bounded, then $(E,\tau)$ has the property (B).

In particular, for an infinite-dimensional Banach space $E$, both
spaces $(E,w)$ and $(E^*,w^*)$ possess the property (B).
\end{proposition}

\begin{cor}
For an infinite-dimensional Banach space $E$, both spaces $(E,w)$
and $(E^*,w^*)$ are not \textsf{US} (hence, are not \textsf{CDH}).
\end{cor}

\section{All \textsf{CDH} topological vector spaces are Baire spaces}

The following theorem and its consequences are our main results.

\begin{thm}\label{CDH_implies_Baire_hom}
Every homogeneous \textsf{CDH} space $X$ containing a copy of the Cantor set is a Baire space.
\end{thm}

Recall that each homogeneous space is either meager or a Baire space. Therefore, the above theorem follows immediately from Lemmas \ref{seq_crowded} and \ref{lemma_meager_not_CDH} (below).

Since a connected \textsf{CDH} space is
homogeneous (see \cite{FL}), we have the following

\begin{cor}\label{CDH_implies_Baire_conn}
Every connected \textsf{CDH} space $X$ with a copy of the Cantor set is a Baire space.
\end{cor}

The next corollary is an instant consequence of homogeneity of topological groups

\begin{cor}\label{CDH_implies_Baire_group} 
Every \textsf{CDH} topological group $G$ containing a copy of the Cantor set is a Baire space. 
\end{cor}

For the third corollary we need the trivial observation, that, for a topological vector space $E$, either $E=\{0\}$ (so it is a Baire space) or $E$ contains a copy of the real line, hence also of the Cantor set.

\begin{cor}\label{CDH_implies_Baire_tvs} 
every \textsf{CDH} topological vector space $E$ is a Baire space.
\end{cor}

Now, let us present two lemmas needed to conclude Theorem \ref{CDH_implies_Baire_hom}. 

\begin{lem}\label{dense_notgdelta} Let $C$ be a copy of the Cantor set in a separable space $X$.
Then the space $X$ contains a countable dense
subset $D$ such that
$D\cap C$ is dense in $C$.
Therefore, $D$ is not a $G_\delta$-subset of $X$.
\end{lem}
\begin{proof}
Let $A$ be any countable dense subset of $X$, and $B$ any countable dense subset of $C$. Put $D = (A\setminus C)\cup B$. Clearly, $D$ is as required.
\end{proof}

\begin{lem}\label{lemma_meager_not_CDH}
Let $X$ be a meager space containing a copy of the Cantor set and a sequentially crowded countable dense subset. Then the space $X$ is not \textsf{CDH}.
\end{lem}
\begin{proof}
The first part of the proof will follow closely the argument from the proof of Lemma \ref{lemma_B_not_US}.
Striving for a contradiction, suppose that $X$ is \textsf{CDH}.
Let
$\{A_n:n=1,2\ldots\}$ be a countable family of closed nowhere dense
subsets of $X$ such that $X= \bigcup_{n\ge 1} A_n$. Without loss of
generality, we can assume that for every $n$ we have $A_n\subseteq
A_{n+1}$.

Since $X$ is \textsf{CDH}, every countable dense subset of $X$ is sequentially crowded.
Fix a countable dense set $D_0\subseteq X$ and let
$D_0=\{d_1,d_2,\ldots\}$ be a faithful enumeration of $D_0$. Since
the set $A_n$ is closed and has empty interior, for each $n\geq 1$,
there is a countable set $D_n\subseteq X\setminus A_n$ dense in $X$.
The set $D_n\cup\{d_n\}$ is sequentially crowded being countable
dense in $X$, hence there is a sequence $(x^k_n)_{k\in \omega}$ of
elements of $D_n$ converging to $d_n$. Define
$$D=\bigcup_{n=1}^\infty\{x^k_n:k\in \omega\}.$$
The countable set $D$ is dense in $X$, because its closure contains
$D_0$.

Now, let $C\subseteq X$ be a copy of a Cantor set. By Lemma \ref{dense_notgdelta}, $X$ contains a countable dense
subset $B$, such that $B\cap C$ is dense in
$C$. Since $X$ is \textsf{CDH}, there exists a homeomorphism $h$ of
$X$ with $h(B)=D$. Then, $h(B\cap C)=h(C)\cap D$; hence, $h(C)\cap D$ is
dense in $C'=h(C)$, yet another copy of the Cantor set. To obtain a
contradiction, we will show that for each $C'$, a copy of the Cantor set,
$C'\cap D$ is not dense in $C'$. By the Baire Category Theorem,
there is $n_0$ such that $C'\cap A_{n_0}$ has a nonempty interior in
$C'$. The set $D\cap A_{n_0}$, being a subset of
$\bigcup_{n=1}^{n_0-1}\{x^k_n:k\in \omega\}$, has only finitely many
accumulation points; hence, $C'\cap D$ is not dense in $C'$.
\end{proof}

Let us note that, for metrizable spaces, the assertion of Theorem
\ref{CDH_implies_Baire_hom} is known (see \cite[Corollary 3]{Me} or cf. remarks preceding Proposition
2.2 in \cite{HG}). For the sake of self-completeness we include the
sketch of proof. If a separable metrizable homogeneous space were not a Baire space then,
as pointed out below Theorem
\ref{CDH_implies_Baire_hom}, it would be meager. This would allow to
construct a $G_\delta$ countable dense subset which, in turn, by
\textsf{CDH}, yields that all countable subsets would be $G_\delta$
(see \cite[Theorem 3.4]{FZ}). The latter contradicts our Lemma \ref{dense_notgdelta} (see also \cite[Theorem 3.5]{FZ}).
\medskip

Recall the following result, proved by Fitzpatrick and Zhou in \cite{FZ}:
\begin{thm}\cite[Theorem 2.5]{FZ}\label{Fitzpatrick_Zhou}
Any \textsf{CDH} space $X$ can be represented as a countable union of a discrete family  of clopen homogeneous \textsf{CDH} subspaces $U_n$.
\end{thm}

Combining the above theorem with our Theorem \ref{CDH_implies_Baire_hom}, we get the following:
\begin{thm}
Every \textsf{CDH} space $X$ is a countable union of a discrete family  of homogeneous clopen \textsf{CDH}
subspaces $U_n$, such that each $U_n$ is either a Baire space or does not contain a copy of the Cantor set.
In particular, $X$ is a union of disjoint clopen \textsf{CDH} subspaces $U$ and $V$ such that $U$ is a Baire space,
$V$ is meager and does not contain a copy of the Cantor set. 
\end{thm}
\begin{proof}
The first assertion is a direct consequence of Theorems \ref{Fitzpatrick_Zhou} and \ref{CDH_implies_Baire_hom}. To get the second assertion, we take
$$U=\bigcup\{U_n:U_n\; \text{is Baire}\}\quad \text{and}\quad V=\bigcup\{U_n:U_n\; \text{is meager}\}$$
\end{proof}

The following fact is probably known.
\begin{prp}\label{weak_not_Baire}
Let $E$ be an infinte-dimensional normed space and $F$ be a vector
subspace of the dual $E^\ast$ separating points of $E$. Then the
space $E$ equipped with the weak topology generated by $F$ is not a
Baire space.
\end{prp}
\begin{proof} Let $A = \bigcap\{\varphi^{-1}([-\|\varphi\|,\|\varphi\|]): \varphi\in F\}$.
Clearly $A$ is closed in the weak topology generated by $F$. Using
the facts that $F$ separates the points of $E$ and each weak
neighborhood of $0$ contains a nontrivial vector subspace, one can
easily verify that $A$ has empty (weak) interior. Since the unit
ball of $E$ is contained in $A$, we have $E = \bigcup_{n\in \omega}
nA$. This shows that $E$ with the weak topology is meager,
hence not a Baire space.
\end{proof}

From this fact and Corollary \ref{CDH_implies_Baire_tvs} we immediately
obtain

\begin{cor}\label{weak_not_CDH}
Let $E$ be an infinite-dimensional normed space and $F$ be a linear
subspace of the dual $E^\ast$ separating points of $E$. Then the
space $E$ equipped with the weak topology generated by $F$ is not
\textsf{CDH}.
\end{cor}

Corollary \ref{weak_not_CDH} does not extend to non-normed vector
spaces. For the product space $\mathbb{R}^\omega$, which is
\textsf{CDH}, its weak topology coincides with the product topology.
\section{Cartesian products with \textsf{CDH}}

In the paper \cite[Theorem 3.1]{HZA} it was proved that, for a separable metrizable space $X$, if the product $X^\omega$ is \textsf{CDH}, then $X$ is a Baire space.
Using our Lemma \ref{lemma_meager_not_CDH} we can generalize this result as follows

\begin{thm}\label{product_Baire}
For a separable space $X$ (not necessarily metrizable) and an infinite cardinal $\kappa$, if $X^\kappa$ is \textsf{CDH}, then $X$ and $X^\kappa$ are Baire spaces.
\end{thm}

In the proof we will use two additional easy lemmas, the first one is probably known.

\begin{lem}\label{dichotomy_for_products}
For any space $X$ and an infinite cardinal $\kappa$, the product $X^\kappa$ is either meager or a Baire space.
\end{lem}
\begin{proof} Suppose that $X^\kappa$ is not meager.
Then the Banach Category Theorem (see \cite[Theorem 1.6]{HM}) gives us a nonempty open subspace $U\subset X^\kappa$ which, in its relative topology, is a Baire space.
Since an open subset of Baire space is again a Baire space, without loss of generality we can assume that $U$ is a basic open set, i.e.,
$U = \pi_A^{-1}(V)$, where $A$ is a finite subset of $\kappa$, $V$ a nonempty open subset of $X^A$, and $\pi_A: X^\kappa\to X^A$ is the projection.
Then, for the open continuous projection $\pi_{\kappa\setminus A}: X^\kappa\to X^{\kappa\setminus A}$ we have
$\pi_{\kappa\setminus A}(U) = X^{\kappa\setminus A}$, hence $X^{\kappa\setminus A}$ is a Baire space,
cf. \cite[Corollary 4.2]{HM}.
Obviously, $X^{\kappa\setminus A}$ is homeomorphic to $X^\kappa$.
\end{proof}

\begin{lem}\label{copies_of_Cantor}
Let $X$ be a space of cardinality $|X|>1$, and $\kappa$ be an infinite cardinal. Then each point  $x\in X^\kappa$ is contained in a copy of the Cantor set $C\subset  X^\kappa$.
\end{lem}
\begin{proof} Let $x = (x_\alpha)_{\alpha<\kappa}$. For each $n\in\omega$, pick a subset $C_n\subset X$ of size $2$, such that $x_n\in C_n$. For $\omega\le \alpha <\kappa$ put $C_\alpha = \{x_\alpha\}$. It is clear that the set $C = \prod_{\alpha<\kappa} C_\alpha$ is as required.
\end{proof}

\begin{proof}[Proof of Theorem \ref{product_Baire}]
Assume that the product $X^\kappa$ is \textsf{CDH}. Clearly, we can assume that $|X|>1$. We will show that $X^\kappa$ contains a sequentially crowded countable dense subset $D$. Let $A$ be any countable dense subset of $X^\kappa$. For each $a\in A$, use Lemma \ref{copies_of_Cantor} to find a copy $C_a\subset X^\kappa$ of the Cantor set, containing $a$.
Let $D_a$ be a countable dense subset of $C_a$. One can easily verify that the set $D = \bigcup\{D_a: a\in A\}$ has the required properties.

Now, by Lemmas \ref{lemma_meager_not_CDH}  and \ref{copies_of_Cantor}, $X^\kappa$ is not meager. Therefore, Lemma \ref{dichotomy_for_products} implies that $X^\kappa$ is a Baire space. Since $X$ is an image of $X^\kappa$ under an open continuous projection, it is also a Baire space.
\end{proof} 

R.\ Hern\'andez-Guti\'errez established in \cite{HG1} the following result concerning Cartesian products with \textsf{CDH} (we are grateful to him for turning our attention to this result):

\begin{thm}\cite[Theorem 2.3]{HG1}\label{product_HG}
Let $X$ and $Y$ be two crowded spaces of countable $\pi$-weight. If $X\times Y$ is \textsf{CDH}, then $X$ contains a copy of the Cantor set if and only if
$Y$ contains a
copy of the Cantor set.
\end{thm}

Using methods developed above, we can generalize this theorem by eliminating the assumption that spaces under consideration have countable $\pi$-weight
(see \cite[1.4]{Tk} for the definition of $\pi$-weight),
i.e. we have:

\begin{thm}\label{product_Cantor_set}
Let $X$ and $Y$ be two crowded spaces. If $X\times Y$ is \textsf{CDH}, then $X$ contains a copy of the Cantor set if and only if $Y$ contains a
copy of the Cantor set. 
\end{thm}

We will need the following lemma.

\begin{lem}\label{lem_1-1_projection}
Let $U$ be an open homogeneous \textsf{CDH} subset of the product $X\times Y$ of two spaces $X$ and $Y$. If $Y$ is crowded and $U$
contains a nontrivial convergent sequence,
then $U$ contains a countable dense subset
$E$ such that for every $y\in Y$ the intersection $E\cap (X\times \{y\})$ has at most one element.
\end{lem}

\begin{proof}
Let $\pi_Y:X\times Y\to Y$ be the projection map. By Lemma \ref{seq_crowded}, the set $U$ contains a countable dense subset
$D=\{d_0,d_1,\ldots\}$ which is sequentially crowded. Since $U$ is \textsf{CDH}, every countable dense subset of $U$ is sequentially crowded.
We shall construct $E$ inductively.
The fiber $X_0=(X\times\{\pi_Y(d_0)\})\cap U$ is a closed nowhere-dense subset of $U$ (because $U$ is open in $X\times Y$ and $Y$ has no
isolated points).
Thus, the set $(D\setminus X_0)\cup\{d_0\}$ is dense in $U$, hence, sequentially crowded. We infer that there is
a sequence $(e^k_0)_{k\in\omega}$ of elements of $D\setminus X_0$ convergent to $d_0$. Passing to a subsequence we can ensure that all points in the sequence
$(\pi_Y(e^k_0))_{k\in\omega}$ are distinct.

Now, suppose that for $n\geq 1$ and every $m<n$, we have constructed sequences $(e^k_m)_{k\in\omega}$ such that:
\begin{enumerate}
 \item $\pi_Y\upharpoonright \bigcup_{m< n}\{e_m^k:k\in\omega\}$ is one-to-one, and
 \item the sequence $(e^k_m)_{k\in\omega}$ converges to $d_m$.
\end{enumerate} 
We will construct a sequence $(e^k_n)_{k\in\omega}$ in such a way that the conditions (1) and (2) above will remain true for $m=n$. To this end, consider
the set
$$X_n=U\cap (X\times \bigcup_{m<n}\pi_Y(\{e^k_m:k\in\omega\}\cup\{d_m\})).$$
The set $\{e^k_m:k\in\omega\}\cup\{d_m\}$ is compact being a convergent sequence, hence $X_n$ is closed in $U$. Moreover, the set $X_n$ is nowhere-dense.
Indeed, if $V$ were a nonempty open in $U$ subset of $X_n$, it would be open in $X\times Y$ (since $U$ is). The projection $\pi_Y:X\times Y\to Y$ is open, so
the set $$\pi_Y(V)\subseteq \bigcup_{m<n}\pi_Y(\{e^k_m:k\in\omega\}\cup\{d_m\})$$ would be open in $Y$. But $Y$ has no isolated points and
$\bigcup_{m<n}\pi_Y(\{e^k_m:k\in\omega\}\cup\{d_m\})$ is a finite union of convergent sequences, a contradiction.

Since $X_n$ is closed and nowhere-dense, the set $(D\setminus X_n)\cup\{d_n\}$ is dense in $U$ and thus it is sequentially crowded.
It follows that there is a sequence $(e^k_n)_{k\in\omega}$ of elements of $D\setminus X_n$,
convergent to $d_n$. Again, passing to a subsequence we can make sure that all points in the sequence $(\pi_Y(e^k_n))_{k\in\omega}$ are distinct.
This finishes the inductive construction of sequences $(e^k_n)_{k\in \omega}$, for $n=0,1,\ldots$, satisfying conditions (1) and (2).

We define $$E=\bigcup_{n=0}^\infty\{e_n^k:k\in\omega\}.$$
The set $E$ is dense in $U$, because by (2), its closure contains $D$. It follows from (1) that $\pi_Y\upharpoonright E$ is one-to-one, hence
for each $y\in Y$ the set $(X\times\{y\})\cap E$ has at most one element.
\end{proof}

\begin{proof}[Proof of Theorem \ref{product_Cantor_set}]
Let $\pi_Y:X\times Y\to Y$ be the projection map.
By symmetry it is enough to show that if $X$ contains a copy of the Cantor set, then $Y$ contains a copy of the Cantor set too. 
Suppose that $X$ contains a copy of the Cantor set. Then, clearly, the product $X\times Y$ contains a copy of the Cantor set $C'$.
By Theorem \ref{Fitzpatrick_Zhou}, there is a clopen homogeneous \textsf{CDH} space $U$ which meets $C'$.
Since $U$ is clopen, the intersection $C = U\cap C'$ is again a copy of the Cantor set. By Lemma \ref{dense_notgdelta}, there is a countable dense subset $D$ of $U$
such that $D\cap C$ is dense in $C$. Applying Lemma \ref{lem_1-1_projection} to $U$, we infer that there is a countable dense subset $E$ of $U$ such that
the map $\pi_Y\upharpoonright E$ is one-to-one. Now, since $U$ is \textsf{CDH}, there exists a homeomorphism $h:U\to U$ such that $h(D)=E$. The set
$D\cap C$ is dense in $C$, so the set $F=E\cap h(C)$ is dense in $h(C)$, yet another copy of the Cantor set. It follows that $F$ is crowded and
$\pi_Y\upharpoonright F$ is one-to-one. Obviously, $\pi_Y(F)$ is crowded and it is dense in $\pi_Y(h(C))$. Therefore, $\pi_Y(h(C))$ is a crowded compact metrizable
subspace of $Y$, so it must contain a copy of the Cantor set.
\end{proof}

One can easily observe that the above proof actually gives the following.

\begin{thm}\label{product_Cantor_2}
Let $\prod_{t\in T} X_t$ be a \textsf{CDH} product of crowded spaces $X_t$. If $\prod_{t\in T} X_t$ contains a copy of the Cantor set, then each  $X_t$ does so. 
\end{thm}

Since every infinite product of crowded spaces contains a copy of the Cantor set, we immediately obtain the next corollary which generalizes \cite[Corollary 2.5]{HG1}.

\begin{cor}\label{product_Cantor_3} Let $T$ be an infinite set, and $X_t$ be a crowded space for $t\in T$. If the product $\prod_{t\in T} X_t$ is \textsf{CDH}, then each  $X_t$ contains a copy of the Cantor set.
\end{cor}

\section{\textsf{CDH} of $C_p(X)$}\label{cpx}

By Corollary \ref{CDH_implies_Baire_tvs}, we know that if a function space
$C_p(X)$ is \textsf{CDH}, then it is necessarily a Baire space.
Interestingly, there is a topological characterization of the space
$X$ for which the space $C_p(X)$ is a Baire space (see \cite[Theorem
6.4.3]{vM}). The proof of the following simple fact (cf.\ \cite[S.284]{Tk}) does not require
a use of this characterization result, however. First, recall that a
subset $A$ of $X$ is \textit{bounded} if for every $f\in C_p(X)$,
the set $f(A)$ is bounded in $\mathbb{R}$.
Note, that every finite set $A$ is bounded, so this notion is
meaningful if $A$ is infinite.

\begin{proposition}\label{notBaire}
Let $X$ be a space containing an infinite bounded subset $A$. Then $C_p(X)$ is meager. 
\end{proposition}
\begin{proof}
For $n\in \omega$, we set $$F_n=\{f\in C_p(X):f[A]\subseteq
[-n,n]\}.$$ Clearly, the sets $F_n$ are closed and have empty
interiors in $C_p(X)$ because $A$ is infinite. Moreover, since $A$
is a bouded subset of $X$, we have $\bigcup_{n\in \omega}
F_n=C_p(X)$.
\end{proof}

Combining the proposition above with Corollary \ref{CDH_implies_Baire_tvs} we immediately obtain:

\begin{thm}\label{thm_inf_bounded}
Let $X$ be a space containing an infinite bounded subset $A$. Then $C_p(X)$ is not \textsf{CDH}.
\end{thm}

Obviously, a nontrivial convergent sequence in a space $X$ forms an infinite bounded subset.  Therefore, we obtain the next corollary which answers
Question 2.6 from \cite{HG}.

\begin{cor}\label{cor_metr}
Let $X$ be a metrizable space. If $C_p(X)$ is \textsf{CDH} then $X$ is discrete.
\end{cor}

In the context of the \textsf{CDH} property of $C_p(X)$ spaces, it is worth recalling that $C_p(X)$ is separable if and only if
$X$ admits a weaker separable metrizable topology (see \cite[S.174]{Tk}). If such $X$ contains additionally an infinite bounded subset,
then it actually contains a nontrivial convergent sequence. Indeed, 
if $A$ is bounded so is its closure $\bar A$. By \cite[Corollary 6.10.9]{AT}, $\bar A$ is compact. Having a weaker
separable metrizable topology $\bar A$ is a metrizable compactum;
therefore, $A$ contains a nontrivial convergent sequence. We are
indebted to Alexander Osipov for informing us about this observation.
Therefore, by the above fact, in Theorem \ref{thm_inf_bounded} (and also in
Proposition \ref{dense_gdelta} below) instead of assuming that $X$ contains an infinite bounded subset, one can equivalently assume that $X$ contains a nontrivial
convergent sequence.
\medskip

For a discrete space $X$, we have $C_p(X) = \mathbb{R}^X$.
The following characterization of \textsf{CDH} products of the real line incorporates assertions that can be found in
\cite{SZ} and \cite{HZA}, cf. \cite[Section 6]{HG}.

\begin{thm}\label{characterization_products}
The product $\mathbb{R}^X$ (\,$[0,1]^X,\ \{0,1\}^X$) is \textsf{CDH} if and only if $X$ is of cardinality less than the
pseudointersection number $\mathfrak{p}$.
\end{thm}

The above result together with our Corollary \ref{cor_metr} yields:

\begin{thm}\label{characterization_metr}
Let $X$ be a metrizable space. Then $C_p(X)$ is \textsf{CDH} if and
only if $X$ is discrete of cardinality less than $\mathfrak{p}$.
\end{thm}

Very recently G.\ Plebanek has shown that, for $X$ of cardinality
$\omega_1$, the product space $\mathbb{R}^X$ is {\it not} a
hereditarily Baire space\footnote{Jan van Mill has found a short proof of this result and has kindly allowed us to include it in our paper. His argument can be found in Appendix.}. Hence, if $X$ is a discrete space of cardinality $\omega_1$ then the space $C_p(X)$ is not a hereditarily Baire space.

This implies that, under an additional set-theoretic assumption that
$\omega_1 < \mathfrak{p}$, the product $\mathbb{R}^{\omega_1}$ is an
example of a topological vector space ($C_p(X)$ space) which is
\textsf{CDH} but is not a hereditarily Baire space. Thus, it is not
possible to replace \textit{Baire} by \textit{hereditarily Baire} in
Corollary \ref{CDH_implies_Baire_tvs}, even for spaces of the form
$C_p(X)$. However, the following question remains open.

\begin{que}
Suppose that $X$ is countable and $C_p(X)$ is \textsf{CDH}. Is $C_p(X)$ a hereditarily Baire space?
\end{que}

Note that by the results of \cite{Ma} the existence of a countable nondiscrete $X$ such that $C_p(X)$ is a hereditarily Baire space  is equivalent to the existence of a nonmeager $P$-filter on $\omega$.

An obvious consequence of Corollary \ref{cor_metr} is the following

\begin{cor}
If $X$ is an uncountable separable metrizable space, then $C_p(X)$ is not \textsf{CDH}.
\end{cor}

The last corollary can be also proved in a different, more direct
way. Let us show the following

\begin{proposition}\label{dense_gdelta}
Let $X$ be a separable space which contains an infinite bounded subset $A$. If $C_p(X)$ is separable, then
it contains a countable
dense subset $E$ which is a $G_\delta$-set in $C_p(X)$. 
\end{proposition}
\begin{proof}
By a standard inductive argument we can pick a sequence $(a_n)$, $a_n\in A$, and a sequence  $(U_n)$ of pairwise disjoint open sets such that $a_n\in U_n$, for $n\in\omega$. For each $n$, take a continuous function $g_n: X\to [0,1]$ such that $g_n(a_n)=1$ and $g_n(x)=0$ for $x\in X\setminus U_n$. Let $\{h_n: n\in\omega\}$ be dense in $C_p(X)$, and let $r_n = |h_n(a_n)| + n$. Put $f_n = h_n + r_ng_n$ and define $E = \{f_n: n\in\omega\}$. We will show that $E$ has the required properties.

Take any finite set $F = \{x_1,\dots,x_k\}\subseteq X$, and a sequence $V_1,\dots,V_k$ of nonempty open subsets of $\mathbb{R}$, and consider the basic open set
$$W = \{f\in C_p(X): f(x_i)\in V_i,\ i\le k\}$$
in $C_p(X)$ given by these sequences. Observe that $W$ contains infinitely many $h_n$, since $C_p(X)$ is crowded. On the other hand $F$ intersects only finitely many $U_n$. Therefore we can find $n_0$ such that $h_{n_0}\in W$ and $F\cap U_{n_0} = \emptyset$. Then $f_{n_0}|F = h_{n_0}|F$, hence  $f_{n_0}\in W$.

Let $B_k = \{f\in C_p(X): \forall a\in A\ |f(a)|\le k\}$. Clearly, all sets $B_k$ are closed and $C_p(X) = \bigcup_{k\in\omega} B_k$.
From the definition of $r_n$ and $f_n$ immediately follows that $|f_n(a_n)|\ge n$, therefore $f_n\notin B_k$ for $n>k$.

Since $X$ is separable, each finite subset of $C_p(X)$ is a $G_\delta$-set (see \cite[S.173]{Tk}). Therefore $$C_p(X)\setminus E = \bigcup_{k\in\omega} (B_k\setminus \{f_n: n\le k\})$$
is an $F_\sigma$-set and $E$ is a a $G_\delta$-subset of $C_p(X)$.
\end{proof}

From Lemma \ref{dense_notgdelta} and Proposition \ref{dense_gdelta} we immediately obtain version of Theorem \ref{thm_inf_bounded} for separable $X$.

\medskip

For a space $X$ and a subset $Y$ of $\mathbb{R}$ we denote by $C_p(X,Y)$ the subspace of $C_p(X)$ consisting of functions with values in $Y$, i.e.,  $C_p(X,Y) = C_p(X)\cap Y^X$. Let $\mathbb{I}$ denote the unit interval $[0,1]$. For the space $C_p(X,\mathbb{I})$ we have the following counterpart of Theorem \ref{characterization_metr}

\begin{thm}\label{characterization_metr_I}
Let $X$ be a metrizable space. Then $C_p(X,\mathbb{I})$ is \textsf{CDH} if and
only if $X$ is discrete of cardinality less than $\mathfrak{p}$.
\end{thm}

For the proof we need a weaker version of Proposition \ref{notBaire} for $C_p(X,\mathbb{I})$

\begin{prp}\label{Cp(X,I)_meager}
Let $X$ be a space containing a nontrivial convergent sequence. Then the space $C_p(X,\mathbb{I})$ is meager.
\end{prp}
\begin{proof}
Let $(x_n)$ be a sequence of distinct points of $X$ converging to a point $x\in X$. For $n\in \omega$, we set $$F_n=\{f\in C_p(X,\mathbb{I}): \forall (k\ge n)\  |f(x_k) - f(x)|\le 1/3\}.$$ Clearly, the sets $F_n$ are closed and $C_p(X,\mathbb{I}) = \bigcup_{n} F_n$, because, for each $f\in C_p(X,\mathbb{I})$, $f(x_n)\to f(x)$. Given $n$, the set $F_n$ 
has empty interior in $C_p(X,\mathbb{I})$, since, for any $f\in C_p(X,\mathbb{I})$ and any finite $A\subset X$, we can find $k\ge n$ such that $x_k\notin A$, and a function $g\in C_p(X,\mathbb{I})$ such that $g|A = f|A$ and $|g(x_k) - g(x)|\ge 1/2$.
\end{proof}

\begin{lem}\label{sep_sep}
For a space $X$, if $C_p(X,\mathbb{I})$ is separable, so is the space $C_p(X,(0,1))$.
\end{lem}
\begin{proof}
Let $A$ be a countable dense subset of $C_p(X,\mathbb{I})$. One can easily verify that
$$D = \{(1-1/n)f + 1/(2n): f\in A,\ n=1,2,\dots\}\,,$$
is a dense subset of $C_p(X,(0,1))$.
\end{proof}

\begin{proof}[Proof of Theorem \ref{characterization_metr_I}]
For a discrete space $X$ we have $C_p(X,\mathbb{I}) = \mathbb{I}^X$, hence the  ``if'' part of the theorem follows directly from Theorem \ref{characterization_products}.

To show the  ``only if'' part, assume, towards a contradiction, that $C_p(X,\mathbb{I})$ is \textsf{CDH} and $X$ is not discrete. Since $X$ contains a nontrivial convergent sequence, by Proposition \ref{Cp(X,I)_meager}, the space $C_p(X,\mathbb{I})$ is meager. Obviously, $C_p(X,\mathbb{I})$ contains a copy of $\mathbb{I}$ (constant functions),
hence also a copy of the Cantor set. By Lemma \ref{sep_sep}, the space  $C_p(X,(0,1))$ is separable.
Since this space is homeomorphic to $C_p(X)$, by Lemma \ref{seq_crowded}, it contains a sequentially crowded countable dense subset $D$.
Moreover, using the fact that the space
$C_p(X,(0,1))$ is dense in $C_p(X,\mathbb{I})$, the set $D$ is also dense in $C_p(X,\mathbb{I})$.
Now, the desired contradiction follows from Lemma \ref{lemma_meager_not_CDH}.
\end{proof}

\begin{remark}
Using a similar method as above, one can also prove that, for a zero-dimensional metrizable space $X$, the space $C_p(X,\{0,1\})$ is \textsf{CDH} if and
only if $X$ is discrete of cardinality less than $\mathfrak{p}$. 
\end{remark}

Recall that $C_p^\ast(X)$ is the subspace of $C_p(X)$ consisting of bounded functions. It is well known that for an infinite space $X$, $C_p^\ast(X)$ is not a Baire space (cf. Proposition \ref{weak_not_Baire}).

\begin{cor}
For no infinite space $X$, the space $C_p^\ast(X)$ is \textsf{CDH}.
\end{cor}
\medskip

Our next theorem and corollary slightly strengthen the results of
Osipov, cf.\ \cite[Sec.\ 8]{HG} and \cite{Os}. Their statements use
the notion of $\gamma$-set; here is one way of
phrasing this notion: A space $X$ is a \textit{$\gamma$-set} if for
any open $\omega$-cover $\mathcal{U}$ of $X$, there is a sequence
$(U_n)_{n\in\omega}$, $U_n\in \mathcal{U}$, such that every point
$x\in X$ belongs to all but finitely many $U_n$'s. Recall that a
cover $\mathcal{U}$ of a space $X$ is an \textit{$\omega$-cover} if,
for any finite subset $F$ of $X$, there exists $U\in \mathcal{U}$
such that $F\subseteq U$.

\begin{thm}\label{Osipov}
Let $X$ be a space such that $X^n$ is hereditary Lindel\"of for every $n\in\omega$. If $X$ is not a $\gamma$-set, then  $C_p(X)$ is not \textsf{US}.
\end{thm}

\begin{cor}\label{cor4.16}
If $X$ is a separable metrizable space which is not a $\gamma$-set, then  $C_p(X)$ is not \textsf{US}.
\end{cor}

Theorem \ref{Osipov} is an immediate consequence of Lemma
\ref{seq_crowded} and the next lemma. The proof of that lemma
follows closely the argument from the proof of the Gerlits-Nagy
theorem stating that a space $X$ is a  $\gamma$-set if and only if
$C_p(X)$ is a Fr\'echet-Urysohn space, cf.\ \cite[Thm.\ II.3.2]{Ar}.

\begin{lem}
Let $X$ be a space such that $X$ is not a $\gamma$-set and $X^n$ is
hereditary Lindel\"of for every $n\in\omega$. Then $C_p(X)$ contains
a countable dense subset which is not sequentially crowded.
\end{lem}
\begin{proof}
Recall that the assumption that all finite powers of $X$ are
hereditary Lindel\"of implies that $C_p(X)$ is hereditary separable
(see \cite[II.5.10]{Ar}).

Let $\mathcal{U}$ be an open  $\omega$-cover of $X$ witnessing the
fact that $X$ is not a $\gamma$-set. Define
$$G = \{f\in C_p(X): \exists (U\in \mathcal{U})\ f^{-1}(\mathbb{R}\setminus\{0\})\subseteq U \}\,.$$
Using the fact that $\mathcal{U}$ is an  $\omega$-cover one can
easily verify that the set $G$ is dense in $C_p(X)$. Take a
countable dense subset $H$ of $G$ and put $D = H\cup\{c_1\}$. We
will prove that $D$ is not sequentially crowded. To this end, we
will show that no sequence $(f_n)$ of functions from $D$, distinct
from $c_1$, converges to $c_1$; here, $c_1$ is a constant function
with value of $1$ . Suppose the contrary. For each $n$ find $U_n\in
\mathcal{U}$ with $f_n^{-1}(\mathbb{R}\setminus\{0\})\subseteq U_n$.
Given $x\in X$, we have $f_n(x)\to 1$, hence $f_n(x) \ne 0$ for $n >
n_0$. It follows that $x\in U_n$ for $n > n_0$, a contradiction with
our assumption on $\mathcal{U}$.
\end{proof}

Corollary \ref{cor4.16} suggests the following:

\begin{que}
Let $X$ be an uncountable separable metrizable space. Is it true that $C_p(X)$ is not \textsf{US}?
\end{que}
\medskip

Let us also comment on Question 2.5 from \cite{HG}. Hern\'andez-Guti\'errez asked if, for an  uncountable Polish  space $X$, the space $C_p(X)$ has a countable dense subset without nontrivial convergent
sequences. An affirmative answer to this question follows immediately from \cite[Cor.\ 3.20]{T1} and the fact that, for such $X$, the space $C_p(X)$ is
hereditary separable (it has a countable network).
\medskip

We conclude this section with two observations related to the characterization of Hern\'andez-Guti\'errez \cite[Theorem 1.2]{HG} of the \textsf{CDH} property of $C_p(N_F)$ spaces that we quoted in the introduction.

\begin{prp}
It is consistent (relative to \textsf{ZFC}) that there exists $2^{2^\omega}$ many \textsf{CDH} spaces  $C_p(N_F)$ which are pairwise non-homeomorphic.
\end{prp}
\begin{proof}
It is well known that the continuum hypothesis or Martin's Axiom implies the existence of  $2^{2^\omega}$ many $P$-ultrafilters $F$ on $\omega$ (cf.\ \cite{Bl}), and each ultrafilter $F$ is nonmeager. By \cite[Theorem 1.2]{HG} all corresponding function spaces $C_p(N_F)$ are \textsf{CDH}. One can easily observe that different filters $F,F'$ generate distinct spaces $C_p(N_F), C_p(N_{F'})$. Indeed, the intersection $\{f\in C_p(N_F): f(\infty)=0\}\cap 2^{N_F}$ is a canonical copy of $F$ in $C_p(N_F)$.  Now, we can repeat the standard counting argument (cf.\ \cite{Ma3}). We have $2^{2^\omega}$ different \textsf{CDH} spaces $C_p(N_F)$. For a given
filter $F$, the space $C_p(N_F)$, being a separable metrizable space, can be homeomorphic to no more than $2^\omega$ other spaces $C_p(N_{F'})$.
\end{proof}

The proof of the next fact follows very closely the argument from the proof of Lemma 3.2 in \cite{Ma}, for the reader's convenience we repeat it here.

\begin{prp}
Let $(F_n)$ be a sequence of nonmeager $P$-filters on $\omega$.
Then the product $\prod_{n\in\omega} C_p(N_{F_n})$ is \textsf{CDH}.
\end{prp}
\begin{proof}
First, recall that for every filter $F$ the space $C_p(N_F)$
is homeomorphic to the space $c_F=\{f\in C_p(N_F): f(\infty)=0\}$, see
\cite[Lemma 2.1]{Ma1}. Hence the space $\prod_{n\in\omega} C_p(N_{F_n})$ is
homeomorphic to the product $\prod_{n\in\omega} c_{F_n}$. We can consider the
product $\prod_{n\in\omega} F_n$ as a
filter $F$ on $\omega\times\omega$. To do this, we identify $(A_n)\in \prod_{n\in\omega} F_n$ with
$\bigcup\{A_n\times\{n\}: n\in\omega\}\subset\omega\times\omega$.
By \cite[Corollary 2.4]{Ma} (cf.  \cite[p.\ 327, Fact 4.3]{Sh}), $F$ is a hereditary Baire space, hence a nonmeager $P$-filter. Therefore, by \cite[Theorem 1.2]{HG}, $C_p(N_F)$ and $c_F$ are \textsf{CDH}. One can easily
verify that $c_F$ is homeomorphic to $\prod_{n\in\omega} c_{F_n}$.
\end{proof}

\section{Bernstein-like direct sum decompositions of Roman Pol}

In this section we show that there are examples of hereditarily Baire
m.l.s. which are not \textsf{CDH}.
Let us recall the notion introduced by R.\ Pol in \cite{Po}. Let $E$ be an infinite-dimensional separable complete m.l.s..
A direct sum decomposition $E = V_1\oplus V_2$ is called a  \emph{Bernstein-like direct sum decomposition} of $E$,
if every linearly independent Cantor set $C$ (i.e., a topological copy of the Cantor set) in $E$ intersects both subspaces $V_i$.
R.\ Pol showed that every infinite-dimensional separable complete m.l.s. admits a Bernstein-like direct sum decomposition.
We will also use Lemma 3.1 from \cite{Po}:

\begin{lem}\label{lem_Pol}
Let $A$ be an analytic set in a separable complete m.l.s. $E$.
If $A$ contains an uncountable linearly independent set, then A contains a linearly independent Cantor set.
\end{lem}

\begin{lem}\label{closed_Q}
Let $Q$ be a closed copy of the rationals $\mathbb{Q}$ in a topological vector space $V$.
Then, for any nonempty relatively open subset $U$ of $Q$, the linear subspace $\lin U$ (that is, a vector subspace of $V$ spanned by $U$) is infinite-dimensional.
\end{lem}
\begin{proof}
Suppose that, for some nonempty relatively open subset $U$ of $Q$,  $\lin U$ is finite-dimensional.
Pick a nonempty clopen subset $V$ of $Q$ contained in $U$. Then the space $F= \lin V$ is finite dimensional,
hence homeomorphic to some Euclidean space $\mathbb{R}^n$. Since $V$ is a topological copy of $\mathbb{Q}$ which is closed in $F$, we arrive at a contradiction.
\end{proof}

\begin{lem}\label{uncountable_lin}
Let $A$ be a nonempty closed subset of a complete m.l.s. $E$, such that, for any nonempty relatively open subset $U$ of $A$, $\lin U$
is infinite-dimensional. Then $A$ contains an uncountable linearly independent set.
\end{lem}
\begin{proof}
Suppose the contrary, then the subspace $F = \lin A$ is spanned by a countable infinite set $\{x_i: i\in \omega\}$.
Put $F_n = \lin \{x_i: i\le n\},\ n\in \omega$. Each $F_n$ is closed in $F$, and $F = \bigcup_{n\in\omega} F_n$.
By the Baire Category Theorem there is $n_0$ such that $U = \wn_A(F_{n_0}\cap A)\ne \emptyset$.
By our assumption on $A$, $\lin U$ is infinite-dimensional, but it is contained in $F_{n_0}$, a contradiction.
\end{proof}

\begin{prp}\label{Bernstein_HB}
Let $E = V_1\oplus V_2$ be a Bernstein-like direct sum decomposition of a separable infinite-dimensional complete m.l.s. $E$.
Then each summand $V_i$ is a hereditarily Baire space.
\end{prp}
\begin{proof}
Suppose the contrary, then by the Hurewicz theorem the space $V_i$ contains a closed copy $Q$ of the rationals $\mathbb{Q}$.
Let $A$ be a closure of $Q$ in $E$. For any nonempty relatively open subset $U$ of $A$, $\lin(U\cap Q)$ is infinite-dimensional by Lemma \ref{closed_Q}.
Therefore, Lemma \ref{uncountable_lin} implies that $A$ contains an uncountable linearly independent set. Then the set $A\setminus Q$
is analytic and contains an uncountable linearly independent set. From Lemma \ref{lem_Pol}, we infer that $A\setminus Q$ contains a linearly independent
Cantor set $C$. The set $C$ is disjoint from $V_i$, a contradiction.
\end{proof}

\begin{prp}\label{Bernstein_not_CDH}
Let $E = V_1\oplus V_2$ be a Bernstein-like direct sum decomposition of a separable infinite-dimensional complete m.l.s. $E$.
Then each subspace $V_i$ is not \textsf{CDH}.
\end{prp}
\begin{proof}
Using a countable base of $V_i$ and fact that each nonempty open subset $U$ of $V_i$ has infinite-dimensional $\lin U$,
one can easily inductively construct a countable dense subset $D$ of $V_i$ which is linearly independent. By Lemma \ref{dense_notgdelta}
it is enough to show that for any copy $K$ of the Cantor set in $V_i$, $K\cap D$ is not dense in $K$. Suppose that $K\cap D$ is dense in $K$.
Then, for any nonempty relatively open subset $U$ of $K$, $U\cap D$ is infinite, so $\lin U$ is infinite-dimensional.
Hence, by Lemma \ref{uncountable_lin}, $K$ contains an uncountable linearly independent set. Now, by Lemma \ref{lem_Pol}, the latter contains a linearly
independent Cantor set $C$. Then $C$ is disjoint from $V_j$, for $j\ne i$, a contradiction.
\end{proof}

\begin{question}
Does there exist (in ZFC) a noncomplete CDH pre-Hilbert space?
\end{question}

\appendix
\section{$\mathbb{R}^{\omega_1}$ is not a hereditarily Baire space}

The aim of this appendix is to give a proof of a recent (unpublished) result by Plebanek that we quoted in Section \ref{cpx}. This short and elegant proof is due to Jan van Mill, whom we thank for allowing us to include it in this paper.

\begin{thm}[G.\ Plebanek]
The product $\mathbb{R}^{\omega_1}$ is not a hereditarily Baire space.
\end{thm}

\begin{proof}[Proof (van Mill)]
Since the space of natural numbers $\omega$ is a closed subspace of the real line $\mathbb{R}$ it is enough to prove that he product $\omega^{\omega_1}$ contains a closed subset $A$ which is not a Baire space. Let $A = \{f:\omega_1\to\omega: f \text{ is not decreasing}\}$. Since $f\in A$ if and only if, for all $\alpha < \beta <\omega_1$ we have $f(\alpha)\le f(\beta)$, the subspace $A$ is clearly closed. If $f\in \omega^{\omega_1}$ is unbounded, then it is unbounded on some countable interval $[0,\alpha)$, so by looking at the value $f(\alpha)$ we can examine that $f\notin A$. Hence we can write $A$ as $\bigcup_{n}A_n$, where $A_n = \{f\in A: f\le n\}$. Obviously, each $A_n$ is a closed subset of $A$. It remains to observe that it has an empty interior. Indeed, let $V$ be a nonempty open set in $\omega^{\omega_1}$ intersecting $A_n$. Without loss of generality we can assume that $V$ is a basic open set, i.e. consists of functions which have a fixed restriction to some finite set $F\subset\omega_1$. Let $\alpha = \max F$. Take any function $f\in V\cap A_n$, and define $g:\omega_1\to\omega$ in the following way: $g$ agrees with $f$ on $[0,\alpha]$ and takes value $n+1$ on $(\alpha,\omega_1)$. One can easily verify that $g\in (V\cap A)\setminus A_n$, witnessing the fact that $V\not\subset A_n$.
\end{proof}

\subsection*{Acknowledgements}
We would like to thank R.\ Hern\'andez-Guti\'errez, S.V.\ Med\-vedev, J.\ van Mill and A.V.\ Osipov for several valuable comments.

\end{document}